\documentclass[11pt]{article}

\usepackage[a4paper,vmargin={25mm,25mm},hmargin={25mm,25mm}]{geometry}
\usepackage{amsmath,amssymb,amsthm,mathtools}
\usepackage{enumitem}
\usepackage{hyperref}
\usepackage{microtype}

\hypersetup{colorlinks=true,linkcolor=blue,citecolor=blue,urlcolor=blue}

\newtheorem{theorem}{Theorem}[section]
\newtheorem{proposition}[theorem]{Proposition}
\newtheorem{lemma}[theorem]{Lemma}

\theoremstyle{definition}
\newtheorem{definition}[theorem]{Definition}
\newtheorem{remark}[theorem]{Remark}

\newcommand{\R}{\mathbb R}
\newcommand{\Z}{\mathbb Z}
\newcommand{\T}{\mathbb T}
\newcommand{\cH}{\mathcal H}
\newcommand{\cL}{\mathcal L}
\newcommand{\disc}{\operatorname{disc}}
\newcommand{\diam}{\operatorname{diam}}
\newcommand{\supp}{\operatorname{supp}}
\newcommand{\BVHK}{\operatorname{BV}_{\mathrm{HK}}}

\newcommand{\norm}[1]{\left\|#1\right\|}

\title{Polylogarithmic Full-Chord Buffon Discrepancy}
\author{Samuel Korsky}
\date{May 18, 2026}

\begin{document}
\maketitle

\begin{abstract}
\noindent
Steinerberger introduced the Buffon discrepancy problem, asking how accurately a one-dimensional set of length $L$ in a convex body can match the Crofton-predicted line-intersection counts, and proved an $O\left(L^{1/3}\right)$ upper bound via a Steinhaus longimeter construction. Using the Aistleitner--Bilyk--Nikolov arbitrary-measure star-discrepancy theorem we demonstrate the existence of full-chord constructions with discrepancy $O\left((\log L)^{3/2}\right)$ for every fixed compact convex body with finite piecewise $C^2$ boundary.  In the disk, we prove that every full-chord construction has discrepancy at least $\Omega\left(\log L\right)$, using Schmidt's two-dimensional rectangle discrepancy lower bound.
\end{abstract}

\section{Introduction}

Let $\Omega\subset\R^2$ be a compact convex body with nonempty interior.  If $S\subset\Omega$ is a finite union of $C^1$ arcs, counted without multiplicity, write
\[
        L(S)=\cH^1(S).
\]
For a line $\ell$, let $\#(\ell\cap S)$ denote the number of transverse intersections, counted with multiplicity over the arcs.  The exceptional lines which are tangent to an arc, pass through an endpoint, or contain a subarc are ignored by taking the essential supremum in line space.  We define
\[
        \disc_\Omega(S)
        :=
        \operatorname*{ess\,sup}_{\ell\in\cL}
        \left|
        \#(\ell\cap S)-
        \frac{2L(S)}{\pi|\Omega|}\cdot\cH^1(\ell\cap\Omega)
        \right|.
\]
A \emph{full chord} of $\Omega$ is a set of the form
\[
        g\cap\Omega,
\]
where $g$ is a line meeting $\Omega$.  In this note, a full-chord construction is a finite union of such full chords.

\bigskip
\noindent
This is the same geometric class used by the Steinhaus-type constructions in Steinerberger's original work~\cite{SteinerbergerBuffon} and in the randomly shifted construction of~\cite{KorskyRandomShifted}: in both cases one chooses lines in the plane and restricts them to $\Omega$, thereby producing full chords. The present paper differs in how those chords are chosen. Rather than using evenly spaced directions together with deterministic or randomly shifted one-dimensional lattices of offsets, we sample directly in the two-dimensional endpoint-pair space of chords. This viewpoint is useful because a full chord is determined by its two boundary endpoints, and a test chord intersects it exactly when those endpoints lie on opposite boundary arcs. 

\bigskip
\noindent
Thus the problem becomes a two-dimensional discrepancy problem, yielding a polylogarithmic upper bound for general convex bodies and, in the disk, a logarithmic lower bound. We note that Steinerberger's bounded-discrepancy construction for the disk uses concentric circles and is therefore not a full-chord construction.

\begin{theorem}[Polylogarithmic upper bound for full chords]\label{thm:upper}
Let $\Omega\subset\R^2$ be a compact convex body with nonempty interior and finite piecewise $C^2$ boundary.  There is a constant $C_\Omega<\infty$ such that for every sufficiently large $L$ there exists a finite union $S$ of full chords of $\Omega$ satisfying $L(S)=L$ and
\[
        \disc_\Omega(S)
        \le C_\Omega(\log L)^{3/2}.
\]
\end{theorem}

\begin{theorem}[Full-chord lower bound in the disk]\label{thm:lower}
There is an absolute constant $c>0$ such that every finite union $S$ of full chords of the unit disk $B=B(0,1)$ satisfies, for all sufficiently large $L(S)$,
\[
        \disc_B(S)\ge c\log L(S).
\]
\end{theorem}

\begin{remark}
The exponents in Theorems~\ref{thm:upper} and~\ref{thm:lower} are unlikely to be sharp.  The purpose of the note is to show that, in the full-chord model, the correct scale is at worst polylogarithmic and at least logarithmic.  This sharply contrasts with constructions based on evenly spaced directions and shifted line grids, whose natural errors are polynomial in $L$.
\end{remark}

\section{Discrepancy Inputs}

The upper bound uses a theorem of Aistleitner--Bilyk--Nikolov~\cite{AistleitnerBilykNikolov}.  We state the support form used below and spell out the harmless support/null-set reduction, since the geometric construction requires the selected points to be genuine endpoint pairs.

\begin{theorem}[Aistleitner--Bilyk--Nikolov, support form]\label{thm:ABN}
Let $\mu$ be a non-atomic Borel probability measure on $[0,1]^2$, and let $Z\subset[0,1]^2$ be a Borel set with $\mu(Z)=0$.  For every $N\ge2$, there are distinct points
\[
        z_1,\dots,z_N\in \supp\mu\setminus Z
\]
such that
\[
        \sup_R\left|\frac1N\sum_{i=1}^N\mathbf 1_R(z_i)-\mu(R)\right|
        \le C\cdot\frac{(\log N)^{3/2}}{N},
\]
where the supremum is over all anchored rectangles $R=[0,u)\times[0,v)$ and $C$ is an absolute constant.  Consequently the same estimate, with another absolute constant, holds for all axis-parallel rectangles.
\end{theorem}

\begin{proof}
The published theorem of Aistleitner--Bilyk--Nikolov gives the same estimate for arbitrary normalized Borel measures on $[0,1]^d$, with exponent $(\log N)^{d-1/2}/N$; for $d=2$ this is $(\log N)^{3/2}/N$.  The only additional point is the choice of representatives.  In the transference/rounding proof, the final points are representatives of positive-$\mu$ cells in a finite Borel partition.  Since $\mu$ is non-atomic and $\mu(Z)=0$, each such cell may be split, if necessary, and its representative may be chosen in $\supp\mu\setminus Z$; every positive-$\mu$ Borel set meets this set.  The representatives can also be chosen distinct, again by non-atomicity.  This does not change the discrepancy estimate, because the proof depends only on the cell masses and on the representative-selection freedom in the final step.  Inclusion-exclusion gives the stated bound for arbitrary axis-parallel rectangles.
\end{proof}

\begin{remark}
The non-atomic and null-set clauses are used only to guarantee that the selected endpoint pairs are genuine, distinct full chords in general position.
\end{remark}

\smallskip
\noindent
The lower bound uses Schmidt's sharp two-dimensional lower bound for rectangle discrepancy.  We state and prove the weighted version needed below by localizing to a small square and applying the unweighted theorem.

\begin{theorem}[Schmidt rectangle lower bound]\label{thm:schmidt-rect}
There is an absolute constant $c_S>0$ such that for every axis-parallel square $Q\subset\R^2$ and every finite point set $P\subset Q$ with $N=\#P\ge2$,
\[
        \sup_{R\subset Q}
        \left|
        \#(P\cap R)-N\cdot\frac{|R|}{|Q|}
        \right|
        \ge c_S\log N,
\]
where the supremum is over all axis-parallel rectangles $R\subset Q$.
\end{theorem}

\begin{theorem}[Weighted Schmidt rectangle lower bound]\label{thm:weighted-schmidt}
Let $Q\subset\R^2$ be a fixed rectangle and let $\rho\in C^1(Q)$ satisfy
\[
        0<c_0\le \rho\le C_0<\infty,
        \qquad
        \norm{\rho}_{C^1(Q)}\le C_0.
\]
Then there are constants $c,M_0>0$, depending only on $Q,c_0,C_0$, with the following property. For every $M\ge M_0$ and every finite point set $P\subset Q$,
\[
        \sup_{R\subset Q}
        \left|
        \#(P\cap R)-M\int_R\rho(x,y)\,dx\,dy
        \right|
        \ge c\log M,
\]
where the supremum is over all axis-parallel rectangles $R\subset Q$.
\end{theorem}

\begin{proof}
Let
\[
        \Delta:=\sup_{R\subset Q}
        \left|
        \#(P\cap R)-M\int_R\rho
        \right|.
\]
We prove $\Delta\ge c\log M$ for large $M$.

\bigskip
\noindent
Choose a point $x_0$ in the interior of $Q$.  For $M$ large, let $Q_M\subset Q$ be the axis-parallel square centered at $x_0$ with side length
\[
        r=A\left(\frac{\log M}{M}\right)^{1/3},
\]
where $A>0$ will be chosen sufficiently small depending only on $c_S$ and $C_0$.  Put
\[
        N_M:=\#(P\cap Q_M).
\]
Since $Q_M$ is a rectangle,
\begin{equation}\label{eq:local-mass-control-schmidt}
        \left|N_M-M\int_{Q_M}\rho\right|\le \Delta.
\end{equation}
If $\Delta\ge (c_S/16)\log M$, we are done.  Hence assume the opposite.  Since $\rho\ge c_0$,
\[
        M\int_{Q_M}\rho
        \ge c_0Mr^2
        =c_0A^2M^{1/3}(\log M)^{2/3}.
\]
The assumed bound on $\Delta$ is negligible compared to this quantity for large $M$.  Therefore
\[
        N_M\asymp Mr^2,
        \qquad
        \log N_M\ge \frac14\cdot\log M
\]
for all sufficiently large $M$.

\bigskip
\noindent
Apply Theorem~\ref{thm:schmidt-rect} inside the square $Q_M$ to the point set $P\cap Q_M$.  There exists a rectangle $R\subset Q_M$ such that
\begin{equation}\label{eq:local-schmidt}
        \left|
        \#(P\cap R)-N_M\cdot\frac{|R|}{|Q_M|}
        \right|
        \ge \frac{c_S}{4}\cdot\log M.
\end{equation}
We now compare the constant-density target $N_M|R|/|Q_M|$ to the weighted target $M\int_R\rho$.  First, by \eqref{eq:local-mass-control-schmidt},
\[
        \left|
        \frac{|R|}{|Q_M|}\cdot N_M
        -
        \frac{|R|}{|Q_M|}\cdot M\int_{Q_M}\rho
        \right|
        \le \Delta.
\]
Second, if
\[
        \bar\rho_{Q_M}:=\frac1{|Q_M|}\int_{Q_M}\rho,
\]
then the $C^1$ bound on $\rho$ implies that $\rho$ varies by at most $O(C_0r)$ on $Q_M$. Therefore
\[
\begin{aligned}
        \left|
        M\bar\rho_{Q_M}|R|-M\int_R\rho
        \right|
        &\le M\int_R |\rho-\bar\rho_{Q_M}| \\
        &\le C C_0 M r |R| \\
        &\le C C_0 M r^3.
\end{aligned}
\]
By the choice of $r$, this is $C C_0A^3\log M$.  Choose $A$ sufficiently small that $C C_0A^3\le c_S/16$.  Combining with \eqref{eq:local-schmidt},
\[
        \Delta\ge \frac{c_S}{4}\log M-\Delta-\frac{c_S}{16}\log M.
\]
Hence
\[
        2\Delta\ge \frac{3c_S}{16}\log M,
\]
which proves the theorem, after adjusting constants.
\end{proof}

\smallskip
\noindent
We also use a Koksma--Hlawka type consequence of rectangle discrepancy.

\begin{definition}[Hardy--Krause variation in dimension two]
For a $C^2$ function $f$ on $[0,1]^2$, define
\[
\begin{aligned}
        \norm{f}_{\BVHK}
        :=&\ |f(1,1)|
        +\int_0^1 |\partial_x f(x,1)|\,dx
        +\int_0^1 |\partial_y f(1,y)|\,dy \\
        &+\int_0^1\int_0^1 |\partial_{xy} f(x,y)|\,dx\,dy.
\end{aligned}
\]
For a general function of bounded Hardy--Krause variation, the same expression is interpreted distributionally.
\end{definition}

\begin{lemma}[Rectangle discrepancy controls bounded-variation integration]\label{lem:KH}
Let $\mu$ be a finite Borel measure on $[0,1]^2$, let $z_1,\dots,z_N\in[0,1]^2$, and put
\[
        \sigma=\sum_{i=1}^N\delta_{z_i}-N\mu.
\]
Assume that
\[
        \sup_R |\sigma(R)|\le \Delta,
\]
where the supremum is over axis-parallel rectangles.  If $f$ has finite Hardy--Krause variation on $[0,1]^2$, then
\[
        \left|\sum_{i=1}^N f(z_i)-N\int f\,d\mu\right|
        \le \norm{f}_{\BVHK}\Delta.
\]
\end{lemma}

\begin{proof}
First suppose $f\in C^2([0,1]^2)$.  For $x,y\in[0,1]$, put
\[
        S(x,y)=\sigma([0,x]\times[0,y]).
\]
The rectangle discrepancy assumption gives $|S(x,y)|\le \Delta$.  The two-dimensional integration-by-parts formula gives
\[
\begin{aligned}
        \int_{[0,1]^2} f\,d\sigma
        &= f(1,1)S(1,1)
        -\int_0^1 \partial_x f(x,1)S(x,1)\,dx \\
        &\quad -\int_0^1 \partial_y f(1,y)S(1,y)\,dy
        +\int_0^1\int_0^1 \partial_{xy}f(x,y)S(x,y)\,dx\,dy.
\end{aligned}
\]
Taking absolute values proves the estimate for smooth $f$.  For a general function of bounded Hardy--Krause variation, the same identity holds distributionally, replacing the derivative terms by the corresponding finite signed variation measures; alternatively one obtains the result by strict approximation in the $BV_{HK}$ topology.
\end{proof}

\section{Endpoint-Pair Space}

Let $\Omega\subset\R^2$ be a compact convex body with nonempty interior and finite piecewise $C^2$ boundary.  Choose a positively oriented arclength parametrization
\[
        \Gamma:\T\to\partial\Omega,
\]
where $\T=\R/\mathcal P\Z$ and $\mathcal P=\cH^1(\partial\Omega)$.
For notational convenience we identify $\T$ with $[0,1)$ by rescaling arclength.

\bigskip
\noindent
For almost every line $g$ meeting $\Omega$, the intersection $g\cap\partial\Omega$ consists of two distinct points.  Choose one of the two possible endpoint orderings in a measurable way and write
\[
        E(g)=(s(g),t(g))\in\T^2.
\]
Because all sets used below are symmetric under swapping the two endpoints, the choice of ordering is immaterial.

\bigskip
\noindent
Let $d\lambda(g)$ denote the usual motion-invariant measure on unoriented lines, normalized so that for every line segment $J$ of length $h$,
\[
        \lambda\{g:g\cap J\ne\varnothing\}=2h.
\]
Put
\[
        \Lambda_\Omega:=\lambda\{g:g\cap\Omega\ne\varnothing\}<\infty,
\]
and define the probability measure $\mu_\Omega$ on $\T^2$ by
\[
        \mu_\Omega:=\frac1{\Lambda_\Omega}\cdot E_\#\bigl(\lambda|_{\{g:g\cap\Omega\ne\varnothing\}}\bigr).
\]
The measure $\mu_\Omega$ is non-atomic.  Let $K_\Omega=\supp\mu_\Omega$.  Let $Z_\Omega\subset K_\Omega$ be the set of endpoint pairs for which the associated line is tangent, lies in a boundary segment, or has an endpoint at one of the finitely many nonsmooth or coordinate-cut parameters.  Then $\mu_\Omega(Z_\Omega)=0$, and every point of $K_\Omega\setminus Z_\Omega$ is the endpoint pair of a genuine full chord of $\Omega$.  These are the only endpoint pairs used in the construction.

\bigskip
\noindent
For $z=(s,t)\in K_\Omega$, let
\[
        w(z)=|\Gamma(s)-\Gamma(t)|.
\]
This is the length of the corresponding full chord.  By Crofton's formula,
\begin{equation}\label{eq:mean-length}
        \int_{K_\Omega} w(z)\,d\mu_\Omega(z)
        =
        \frac{\pi|\Omega|}{\Lambda_\Omega}.
\end{equation}

\begin{lemma}[Endpoint crossing sets]\label{lem:endpoint-crossing}
Let $\ell$ be a line meeting the interior of $\Omega$, and let $I_\ell\subset\T$ be one of the two boundary arcs cut off by the endpoints of $\ell\cap\Omega$.  A full chord with endpoint pair $(s,t)$ intersects $\ell\cap\Omega$ if and only if exactly one of $s,t$ lies in $I_\ell$.  Thus the crossing set in endpoint-pair space is
\[
        A_{I_\ell}=(I_\ell\times I_\ell^c)\cup(I_\ell^c\times I_\ell).
\]
Moreover,
\begin{equation}\label{eq:mu-crossing}
        \mu_\Omega(A_{I_\ell})
        =
        \frac{2\cH^1(\ell\cap\Omega)}{\Lambda_\Omega}.
\end{equation}
\end{lemma}

\begin{proof}
The first assertion is the planar separation property for chords of a convex body: two chords intersect if and only if their endpoints alternate on the boundary. The formula \eqref{eq:mu-crossing} follows because the numerator is the line measure of all lines $g$ whose chord $g\cap\Omega$ intersects the segment $\ell\cap\Omega$. By Crofton's formula for a segment of length $h$,
\[
        \lambda\{g:g\cap(\ell\cap\Omega)\ne\varnothing\}=2h.
\]
Dividing by $\Lambda_\Omega$ proves the claim.
\end{proof}

\begin{lemma}[Bounded variation of chord length]\label{lem:chord-length-bv}
If $\partial\Omega$ is finite piecewise $C^2$, then the chord-length function
\[
        w(s,t)=|\Gamma(s)-\Gamma(t)|
\]
has finite Hardy--Krause variation on $\T^2$, after cutting $\T^2$ into finitely many coordinate rectangles.  In particular, in any fixed coordinate square representation of $\T^2$,
\[
        \norm{w}_{\BVHK}\le C_\Omega.
\]
\end{lemma}

\begin{proof}
Cut $\T$ at the finitely many nonsmooth boundary parameters and at the coordinate cut.  On each open product of two resulting intervals, away from pairs representing the same boundary point, the function $w$ is $C^2$ and its derivatives are bounded by constants depending only on $\Omega$ and the distance to that coincidence set.

\bigskip
\noindent
It remains to check the local models near coincidences.  If $s$ and $t$ lie on the same smooth $C^2$ boundary arc, arclength parametrization gives
\[
        \Gamma(s)-\Gamma(t)=(s-t)\tau(t)+O_\Omega(|s-t|^2),
\]
where $\tau$ is the unit tangent.  Hence
\[
        w(s,t)=|s-t|+O_\Omega(|s-t|^2).
\]
The mixed derivative of $|s-t|$ is a finite signed measure supported on the diagonal, and the error term has integrable mixed derivative.  The same argument treats the periodic diagonal created by the coordinate cut.

\bigskip
\noindent
At a corner or junction parameter, use one-sided arclength coordinates $u,v\ge0$ along the two adjacent boundary pieces.  The model singularity is
\[
        q(u,v)=|u\tau_+ - v\tau_-|,
\]
where $\tau_+$ and $\tau_-$ are the one-sided unit tangents.  If the tangents are distinct, convexity gives a positive angle between the two rays, and
\[
        |\partial_{uv}q(u,v)|\le \frac{C}{\sqrt{u^2+v^2}}
\]
away from the origin, which is integrable on small squares.  If the tangents agree, this reduces to the smooth-diagonal model above.  The $C^2$ error terms on the two adjacent pieces again have integrable mixed derivative.  Flat pieces cause no additional difficulty, since on a flat interval the diagonal model is exactly $|s-t|$.

\bigskip
\noindent
There are only finitely many cuts and junctions.  Across them the traces of $w$ have bounded one-dimensional variation, so the distributional mixed derivative acquires only finitely many finite jump measures.  Thus the two-dimensional mixed variation and the one-dimensional face variations are finite, with total bounded by a constant depending only on $\Omega$.
\end{proof}

\section{The Polylogarithmic Upper Bound}

We now prove Theorem~\ref{thm:upper}.

\begin{proposition}[Full-chord construction with prescribed number of chords]\label{prop:N-upper}
For every $N\ge2$ there is a union $S_N$ of $N$ full chords of $\Omega$ such that
\[
        \disc_\Omega(S_N)
        \le C_\Omega(\log N)^{3/2}
\]
and
\[
        \left|L(S_N)-N\cdot\frac{\pi|\Omega|}{\Lambda_\Omega}\right|
        \le C_\Omega(\log N)^{3/2}.
\]
\end{proposition}

\begin{proof}
Apply Theorem~\ref{thm:ABN} to the probability measure $\mu_\Omega$ on endpoint-pair space, with $Z=Z_\Omega$.  We obtain $N$ distinct genuine endpoint pairs
\[
        z_i=(s_i,t_i)\in K_\Omega\setminus Z_\Omega
\]
such that every axis-parallel rectangle $R\subset\T^2$ satisfies
\begin{equation}\label{eq:rect-disc-upper}
        \left|
        \#\{i:z_i\in R\}-N\mu_\Omega(R)
        \right|
        \le C(\log N)^{3/2}.
\end{equation}
Here intervals on the circle are cut into at most two intervals in $[0,1)$, changing only the absolute constant.

\bigskip
\noindent
Let $S_N$ be the union of the full chords determined by the endpoint pairs $z_i$.  For a test line $\ell$, Lemma~\ref{lem:endpoint-crossing} gives
\[
        \#(\ell\cap S_N)
        =
        \#\{i:z_i\in A_{I_\ell}\}
\]
for all nonexceptional $\ell$.  Since $A_{I_\ell}$ is a union of $O(1)$ axis-parallel rectangles in $\T^2$, \eqref{eq:rect-disc-upper} gives
\[
        \#(\ell\cap S_N)
        =N\mu_\Omega(A_{I_\ell})+O((\log N)^{3/2}).
\]
By \eqref{eq:mu-crossing},
\[
        N\mu_\Omega(A_{I_\ell})
        =
        \frac{2N}{\Lambda_\Omega}\cdot\cH^1(\ell\cap\Omega).
\]

\smallskip
\noindent
It remains to compare $2N/\Lambda_\Omega$ with the normalization using the actual length.  By Lemma~\ref{lem:KH}, Lemma~\ref{lem:chord-length-bv}, and \eqref{eq:rect-disc-upper},
\[
        L(S_N)=\sum_{i=1}^N w(z_i)
        =N\int w\,d\mu_\Omega+O_\Omega((\log N)^{3/2}).
\]
Using \eqref{eq:mean-length}, this is the asserted length estimate.  Therefore, uniformly in $\ell$,
\[
\begin{aligned}
\left|
\frac{2N}{\Lambda_\Omega}\cdot\cH^1(\ell\cap\Omega)
-
\frac{2L(S_N)}{\pi|\Omega|}\cdot\cH^1(\ell\cap\Omega)
\right|
&\le C_\Omega(\log N)^{3/2},
\end{aligned}
\]
since $\cH^1(\ell\cap\Omega)\le\diam(\Omega)$.  Combining the estimates proves the discrepancy bound.
\end{proof}

\begin{lemma}[Exact length correction]\label{lem:exact-length-correction}
There is a constant $m_\Omega<\infty$ with the following property.  Given any finite collection of full chords and any $\Delta\ge0$, one can choose a finite union $E$ of full chords, sharing no nontrivial segment with the prescribed chords or with each other, such that
\[
        L(E)=\Delta
\]
and $E$ contains at most $m_\Omega(\Delta+1)$ chords.
\end{lemma}

\begin{proof}
Choose a direction $v$ which is not parallel to any nondegenerate line segment in $\partial\Omega$; this is possible because there are only finitely many flat boundary pieces.  For a nearby direction $u$, let $\ell_u(a)$ be the length of the chord cut out by the line with direction $u$ and perpendicular offset $a$.  Since $u$ is not parallel to a flat boundary piece, the exposed faces at the two extreme offsets are points, and $\ell_u(a)$ extends continuously to the offset interval with endpoint values $0$.  It attains a positive maximum.

\bigskip
\noindent
By continuity in the direction, there is an open interval $W$ of directions, avoiding the finitely many flat directions, and a number $d_\Omega>0$ such that every $u\in W$ has a chord of each length in $[0,d_\Omega]$.  Given $\Delta$, write
\[
        \Delta=nd_\Omega+r,
        \qquad
        n\in\Z_{\ge0},
        \qquad
        0\le r<d_\Omega.
\]
Choose $n$ distinct directions in $W$ and take one chord of length $d_\Omega$ in each of those directions.  If $r>0$, choose one further direction in $W$ and take a chord of length $r$.  The directions may be chosen generically so that the new chords share no nontrivial segment with each other or with the prescribed finite collection.  Thus lengths add, and the number of added chords is at most $\Delta/d_\Omega+1$.
\end{proof}

\begin{proof}[Proof of Theorem~\ref{thm:upper}]
Let
\[
        \bar w_\Omega:=\frac{\pi|\Omega|}{\Lambda_\Omega},
        \qquad
        \Phi(L):=(\log L)^{3/2}.
\]
By Proposition~\ref{prop:N-upper}, after increasing a constant if necessary, there is a constant $C_0=C_0(\Omega)$ such that, for all large $N$,
\[
        \disc_\Omega(S_N)\le C_0(\log N)^{3/2},
        \qquad
        \left|L(S_N)-N\bar w_\Omega\right|\le C_0(\log N)^{3/2}.
\]
Choose
\[
        N=\left\lfloor \frac{L-K\Phi(L)}{\bar w_\Omega}\right\rfloor
\]
with $K=K_\Omega>4C_0$ fixed. Then $N\asymp_\Omega L$, so $(\log N)^{3/2}\le 2\Phi(L)$ for all large $L$. Hence
\[
        \disc_\Omega(S_N)\le 2C_0\Phi(L),
        \qquad
        \left|L(S_N)-N\bar w_\Omega\right|\le 2C_0\Phi(L).
\]
By the choice of $N$,
\[
        L-K\Phi(L)-\bar w_\Omega\le N\bar w_\Omega\le L-K\Phi(L).
\]
Therefore
\[
        0\le L-L(S_N)\le (K+2C_0+1)\Phi(L)
\]
for all sufficiently large $L$, after using $\bar w_\Omega\le \Phi(L)$. Put $\Delta:=L-L(S_N)$.  Apply Lemma~\ref{lem:exact-length-correction} to the existing chords in $S_N$ and obtain a finite union $E$ of at most $m_\Omega(\Delta+1)$ full chords with $L(E)=\Delta$, chosen so that no new chord shares a nontrivial segment with another chord.  Set
\[
        S:=S_N\cup E.
\]
Then $L(S)=L$.

\bigskip
\noindent
Each added chord changes every nonexceptional line-intersection count by at most one. Also, changing the length by $\Delta$ changes the Crofton target by at most
\[
        \frac{2\Delta}{\pi|\Omega|}\cdot\diam(\Omega).
\]
Thus for sufficiently large $C_\Omega > 0$
\[
\begin{aligned}
        \disc_\Omega(S)
        &\le \disc_\Omega(S_N)+m_\Omega(\Delta+1)+\frac{2\diam(\Omega)}{\pi|\Omega|}\cdot\Delta \\
        &\le C_\Omega\Phi(L).
\end{aligned}
\]
This proves
\[
        \disc_\Omega(S)\le C_\Omega(\log L)^{3/2}.
\]
\end{proof}

\section{The Disk Lower Bound}

We now prove Theorem~\ref{thm:lower}.  The proof uses the special fact that for the disk, endpoint pairs are simply pairs of boundary angles.

\bigskip
\noindent
Let $B=B(0,1)$ and parametrize $\partial B$ by $e^{ix}$, $x\in\T=\R/2\pi\Z$.  A full chord is represented by an unordered pair $\{x,y\}$.  All arcs in this section are taken half-open, after choosing a coordinate interval in $\T$.

\bigskip
\noindent
Fix a full-chord union $S$ and let $F\subset\T$ be the finite set of boundary endpoints of chords in $S$.  An arc is called admissible if its two endpoints do not lie in $F$.  If $I\subset\T$ is admissible, let $\ell_I$ be the chord joining the endpoints of $I$.  Then $\ell_I$ is nonexceptional for $S$, and a full chord $\{x,y\}$ crosses $\ell_I$ iff exactly one of $x,y$ lies in $I$.

\bigskip
\noindent
The essential supremum in the definition of discrepancy controls every nonexceptional line.  Indeed, near a nonexceptional line the intersection count is locally constant, while the Crofton target varies continuously; if the discrepancy at such a line were larger than the essential supremum, the same would hold on a positive-measure neighborhood.

\bigskip
\noindent
For an admissible arc $I$, define
\[
        C(I):=\#\{\text{chords of }S\text{ with exactly one endpoint in }I\}.
\]
If $D=\disc_B(S)$ and $L=L(S)$, then
\begin{equation}\label{eq:cut-disc}
        \left|C(I)-\frac{4L}{\pi^2}\cdot\sin\frac{|I|}{2}\right|
        \le D
\end{equation}
for every admissible arc $I$.  Indeed, the test chord $\ell_I\cap B$ has length $2\sin(|I|/2)$ and $|B|=\pi$.

\bigskip
\noindent
Fix two disjoint boundary arcs $U,V\subset\T$ of small positive length, separated by a fixed positive distance.  For instance, one may take
\[
        U=[0,\eta),
        \qquad
        V=[\pi/3,\pi/3+\eta)
\]
with $\eta>0$ sufficiently small.  If $I\subset U$ and $J\subset V$ are admissible subarcs, let
\[
        N(I,J)
\]
be the number of chords of $S$ with one endpoint in $I$ and the other in $J$.

\begin{lemma}[Rectangle counts from cut counts]\label{lem:rectangle-from-cuts}
For all admissible subarcs $I\subset U$ and $J\subset V$,
\[
        \left|
        N(I,J)-
        \int_I\int_J \frac{L}{2\pi^2}\cdot\sin\left(\frac{y-x}{2}\right)\,dy\,dx
        \right|
        \le 2D.
\]
\end{lemma}

\begin{proof}
Write $I=[a,b)$ and $J=[c,d)$ in cyclic order $a<b<c<d$ after choosing a coordinate interval containing $U$ and $V$.  Let $G=[b,c)$ be the gap between $I$ and $J$.  Since $a,b,c,d\notin F$, the four arcs below are admissible.  For a chord with endpoints in the four regions $I,G,J,$ and the complement, a direct check gives the identity
\begin{equation}\label{eq:cut-identity}
        N(I,J)=\frac12\bigl(C(I\cup G)+C(G\cup J)-C(G)-C(I\cup G\cup J)\bigr).
\end{equation}
Indeed, a chord with one endpoint in $I$ and one in $J$ contributes $1+1-0-0=2$ to the bracket, while every other pair of regions contributes $0$.

\bigskip
\noindent
Apply \eqref{eq:cut-disc} to the four arcs in \eqref{eq:cut-identity}.  The target contribution is
\[
\frac12\cdot\frac{4L}{\pi^2}\cdot
\left(
\sin\left(\frac{c-a}{2}\right)+\sin\left(\frac{d-b}{2}\right)-\sin\left(\frac{c-b}{2}\right)-\sin\left(\frac{d-a}{2}\right)
\right).
\]
Since
\[
        \frac{\partial^2}{\partial x\partial y}\sin\left(\frac{y-x}{2}\right)
        =\frac14\cdot\sin\left(\frac{y-x}{2}\right),
\]
this equals
\[
        \int_a^b\int_c^d \frac{L}{2\pi^2}\cdot\sin\left(\frac{y-x}{2}\right)\,dy\,dx.
\]

\smallskip
\noindent
The four cut errors enter with factor $1/2$, so their total contribution is at most $2D$.
\end{proof}

\begin{proof}[Proof of Theorem~\ref{thm:lower}]
Let $P\subset U\times V$ be the finite point set consisting of the ordered pairs $(x,y)$ corresponding to chords of $S$ with one endpoint $x\in U$ and the other endpoint $y\in V$.  Then, for half-open subarcs $I\subset U$ and $J\subset V$,
\[
        \#(P\cap(I\times J))=N(I,J).
\]
Apply Theorem~\ref{thm:weighted-schmidt} on $Q=U\times V$ to the density
\[
        \rho_0(x,y)=\frac{1}{2\pi^2}\cdot\sin\left(\frac{y-x}{2}\right)
\]
with $M=L$. Since $U$ and $V$ are separated, $\rho_0$ is $C^1$ and bounded above and below by positive constants on $Q$.  The theorem gives
\[
        \sup_{I\subset U,\ J\subset V}
        \left|
        N(I,J)-
        \int_I\int_J \frac{L}{2\pi^2}\cdot\sin\left(\frac{y-x}{2}\right)\,dy\,dx
        \right|
        \ge c\log L
\]
for all large $L$, after restricting the supremum to admissible half-open arcs.  This restriction is harmless: the endpoints in $F$ and the coordinates of points of $P$ form a finite set, the point count is locally constant under perturbations avoiding this set, and the integral target is continuous.

\smallskip
\noindent
Lemma~\ref{lem:rectangle-from-cuts} bounds the same admissible supremum by $2D$.  Hence
\[
        D\ge \frac{c}{2} \cdot \log L.
\]
This proves the theorem.
\end{proof}

\section{Further Comments}

The upper and lower bounds leave a gap between $\log L$ and $(\log L)^{3/2}$.  The exponent $3/2$ comes from applying the Aistleitner--Bilyk--Nikolov arbitrary-measure star-discrepancy theorem in dimension $2$ to the endpoint-pair measure.  For special measures and special range families one may expect sharper constructions.  In the disk, for example, the endpoint-pair measure is explicit and smooth away from the diagonal, and the relevant crossing sets are very simple unions of rectangles.  It is plausible that the upper bound can be improved to $O(\log L)$.

\bigskip
\noindent
The lower bound relies on the full chord restriction; as Steinerberger showed in \cite{SteinerbergerBuffon}, one may achieve discrepancy $O(1)$ in the disk using a non-linear construction. The natural question following these results is whether one can achieve $O(1)$ discrepancy for arbitrary convex bodies if one were to remove the full chord restriction.

\section*{Acknowledgements} The author thanks Stefan Steinerberger for introducing the Buffon discrepancy problem and Victor Reis for helpful discussion regarding improvements to the original Steinhaus longimeter construction (\cite{SteinerbergerBuffon}, \cite{KorskyRandomShifted}). The author also acknowledges GPT-5.5 for assistance in performing the detailed computations and preparing an initial draft of this preprint. The proof idea and the direction of the argument are due to the author.

\end{document}